\documentclass[11pt]{article}
\usepackage{hyperref}
\usepackage{enumerate,fullpage}
\usepackage{amssymb,amsmath,graphicx,amsthm}
\usepackage{color}

\usepackage{amsfonts,amsthm, amsmath}
\usepackage{epsfig}
\usepackage{makeidx}
\usepackage{graphicx,epstopdf}
\DeclareGraphicsExtensions{.ps,.png,.jpg}
\usepackage{enumerate}

\makeindex
\usepackage{subfig}

\newcommand{\ncom}{\newcommand}
\ncom{\ul}{\underline}
\ncom{\beq}{\begin{equation}}
\ncom{\eeq}{\end{equation}}
\ncom{\bea}{\begin{eqnarray*}}
\ncom{\eea}{\end{eqnarray*}}
\ncom{\beqa}{\begin{eqnarray}}
\ncom{\eeqa}{\end{eqnarray}}
\ncom{\nno}{\nonumber}
\ncom{\non}{\nonumber}
\ncom{\ds}{\displaystyle}
\ncom{\half}{\frac{1}{2}}
\ncom{\mbx}{\makebox{.25cm}}
\ncom{\hs}{\mbox{\hspace{.25cm}}}
\ncom{\rar}{\rightarrow}
\ncom{\Rar}{\Rightarrow}
\ncom{\noin}{\noindent}
\ncom{\bc}{\begin{center}}
\ncom{\ec}{\end{center}}
\ncom{\sz}{\scriptsize}
\ncom{\rf}{\ref}
\ncom{\s}{\sqrt{2}}
\ncom{\sgm}{\sigma}
\ncom{\Sgm}{\Sigma}
\ncom{\psgm}{\sigma^{\prime}}
\ncom{\dt}{\delta}
\ncom{\Dt}{\Delta}
\ncom{\lmd}{\lambda}
\ncom{\Lmd}{\Lambda}
\ncom{\Th}{\Theta}
\ncom{\e}{\eta}
\ncom{\eps}{\epsilon}
\ncom{\pcc}{\stackrel{P}{>}}
\ncom{\lp}{\stackrel{L_{p}}{>}}
\ncom{\dist}{{\rm\,dist}}
\ncom{\sspan}{{\rm\,span}}
\ncom{\re}{{\rm Re\,}}
\ncom{\im}{{\rm Im\,}}
\ncom{\sgn}{{\rm sgn\,}}
\ncom{\ba}{\begin{array}}
\ncom{\ea}{\end{array}}
\ncom{\hone}{\mbox{\hspace{1em}}}
\ncom{\htwo}{\mbox{\hspace{2em}}}
\ncom{\hthree}{\mbox{\hspace{3em}}}
\ncom{\hfour}{\mbox{\hspace{4em}}}
\ncom{\vone}{\vskip 2ex}
\ncom{\vtwo}{\vskip 4ex}
\ncom{\vonee}{\vskip 1.5ex}
\ncom{\vthree}{\vskip 6ex}
\ncom{\vfour}{\vspace*{8ex}}
\ncom{\norm}{\|\;\;\|}
\ncom{\integ}[4]{\int_{#1}^{#2}\,{#3}\,d{#4}}
\ncom{\vspan}[1]{{{\rm\,span}\{ #1 \}}}
\ncom{\dm}[1]{ {\displaystyle{#1} } }
\ncom{\ri}[1]{{#1} \index{#1}}

\newtheorem{theorem}{\bf Theorem}[section]
\newtheorem{remark}{\bf Remark}[section]
\newtheorem{proposition}{Proposition}[section]

\newtheoremstyle
    {remarkstyle}
    {}
    {11pt}
    {}
    {}
    {\bfseries}
    {:}
    {     }
    {\thmname{#1} \thmnumber{#2} }

\theoremstyle{remarkstyle}

\begin{document}

\newpage

\begin{center}
{\Large \bf Densities of Inverse Tempered Stable Subordinators and Related Processes With Mellin Transforrm}
\end{center}
\vone
\vone
\begin{center}
{Neha Gupta}$^{\textrm{a}}$, {Arun Kumar}$^{\textrm{a}}$
\footnotesize{
		$$\begin{tabular}{l}
		$^{\textrm{a}}$ \emph{Department of Mathematics, Indian Institute of Technology Ropar, Rupnagar, Punjab - 140001, India}
		
%$^{b}$Cardiff School of Mathematics, Cardiff University, Senghennydd Road,
%Cardiff, CF24 4AG, UK

\end{tabular}$$}
\end{center}
\vtwo
\begin{center}
\noindent{\bf Abstract}
\end{center}
In this article, the infinite series form of the probability densities of tempered stable and inverse tempered stable subordinators are obtained using Mellin transform. Further, the densities of the products and quotients of stable and inverse stable subordinators are worked out. The asymptotic behaviours of these densities are obtained as $x\rightarrow 0^+$. Similar results for tempered and inverse tempered stable subordinators are discussed. Our results provide alternative methods to find the densities of these subordinators and complement the results available in literature.\\
\vtwo

\vone

\section{Introduction}
The subordinators and inverse subordinators are getting increasing attention in recent years due to their applications in time-changed stochastic processes and related areas. For example time-fractional Poisson process \cite{Meerschaert2011}, space-fractional Poisson process \cite{Orsingher2012}, Poisson inverse Gaussian Process and solutions of fractional Cauchy problems \cite{Meerschaert2009} etc.

The integral transforms are useful tools for studying the distributions of the product and sum of independent random variables. Such as Laplace transform maps the
convolution into multiplication and the Mellin transform maps the multiplicative convolution into the multiplication operation see e.g. \cite{Misra1986}. These transforms help in finding the distributions of the sum and products of independent random variables conveniently.
The integral representation and Mellin transform for the density of a $\alpha$-stable subordinator $D_{\alpha}(t)$ is discussed in (\cite{Zolotarev1986},  Theorem 2.2.1, p.70).
The density of the product of two independent random variables have been discussed in \cite{Karlov2013}. Moreover, the densities of the product, quotient, power and scalar multiple of independent stable subordinators and inverse stable subordinator are obtained in terms of the Fox’s $H$-function in \cite{Kataria2018}.

\noindent In this article, we provide the power series representation of densities of the product and quotient of two independent inverse stable subordinators. We also work on the densities of tempered stable, inverse tempered stable subordinator and the first-passage time of the inverse Gaussian process using Mellin transform. These processes have been used as time-changes of Poisson process and Brownian motion (see \cite{Beghin2009, Gupta2020, Kumar2019, Meerschaert2011, Orsingher2012}). For the integral representation of these densities see e.g. (\cite{Alrawashdeh2017, Kumar2015, Vellaisamy2018}. 

The rest of the paper is organized as follows. In Section 2, we recall the Mellin transform, tempered stable and inverse stable subordinators. Section 3 and Section 4 discuss the main results. Section 3 deals with stable subordinator, inverse stable subordinators and related processes. Last Section deals with tempered stable and inverse tempered stable subordinators.

\section{Preliminaries}
In this section, we recall Mellin transform and its inverse. Further, tempered stable and inverse tempered stable subordinators and their basic properties are introduced.
\subsection{The Mellin transform and its inverse}
The Mellin transform $F(s)$, corresponding to the  probability density function (pdf) $f(x)$ for non-negative random variable $X$ (see \cite{TITCHMARSH1937}, p.47 ),
\begin{equation}\label{Mellin_function}
F(s)= \mathcal{M}_{x}[f(x)]= \mathbb{E}[X^{s-1}]= \int_{0}^{\infty} x^{s-1} f(x) dx,
\end{equation}
where $F(s)$ is analytic in the vertical strip $x_1 < \mathcal{R}e(s) < x_2$ for some $x_1$ and  $x_2.$
\noindent The change of variables $x = e^{-u}$ shows that the Mellin transform is closely related to the Laplace and the Fourier transforms. The Mellin transform plays a crucial role in studying the distributions of the products of positive independent random variables. For positive independent random variables $X_1$ and $X_2$ with Mellin transforms $F_1(s)$ and $F_2(s)$, such that $X=X_1X_2$, we have    
\begin{equation}\label{Mellin_product}
H(s)= \mathbb{E}[X^{s-1}]= \mathbb{E}[X_1^{s-1}]\mathbb{E}[X_2^{s-1}]= F_1(s) F_2(s).
\end{equation}
The inversion formula for \eqref{Mellin_function} follows directly from the corresponding formula for the bilateral Laplace transform and is of the form
\begin{equation}
    f(x)= \frac{1}{2\pi i}\int_{c-i\infty}^{c+i\infty} x^{-s}\mathcal{M}_{x}[f(x)] ds, 
\end{equation}
at all point $x\geq 0$ for which $f(x)$ is continuous and where the path of integration is any line parallel to the imaginary axis and lying within the strip of analyticity of $F(s)$, for more introduction see e.g. \cite{Misra1986}.
%By using the inversion of Mellin transform the probability density function the product in \eqref{Mellin_product} can be evaluated  (see \cite{Epstein1948} )as
%\begin{equation}
 %   h(x)= \frac{1}{2\pi i}\int_{c-i\infty}^{c+i\infty} x^{-s} F_1(s) F_2(s) ds.
%\end{equation}
Consider the ratio $Z=\frac{Z_1}{Z_2}$ of two positive independent random variables $Z_1$ and $ Z_2 >0$ with continuous densities $g_{1}(z)$ and $g_2(z)$ respectively. To find the pdf of the random variable $Z$, it can be reduced in the form of product of two independent random variables. Let $G_1(s)$ and $G_2(-s+2)$ be the Mellin transform of $Z_1$ and $1/Z_2$, respectively. Therefore, the Mellin transform of of $Z$ is (see \cite{Epstein1948})
 \begin{align}\label{mellin_quiotent}
  \mathcal{M}_z[g_1(z)g_2(z)](s)= G_1(s)G_2(-s+2).
 \end{align}
 Then pdf $g(z)$ of the ratio $Z$ is 
 $
  g(z)= \frac{1}{2\pi i}\int_{c-i\infty}^{c+i\infty} z^{-s} G_1(s) G_2(-s+2) ds.
$
%%%%%%%%%%%%%
We recall the following theorem from \cite{Oberhettinger1974}, which will be used further.
\begin{theorem}[\cite{Oberhettinger1974} p.7]\label{Asym_mellin}
Suppose that $\tilde{f}(z)$ which is the Mellin transform for the function $f$ is analytic in a left-hand plane,
$\sigma \leq a$, apart from poles at $z = -a_m, \; m = 0, 1, 2,\ldots:$ let the principal part of the Laurent expansion of $\Tilde{f}(z)$ about $z = -a_m$ be given by
$$
\sum_{n=0}^{N(m)}A_{mn}\frac{(-1)^n n!}{(z+a_m)^{n+1}}.
$$
Assume that $|\tilde{f}(\sigma+i\tau )| \rightarrow 0$ as $|\tau | \rightarrow \infty$ for a $0 \leq \sigma \leq a$, and that $|\tilde{f}(a'+i\tau )$ is integrable for $|\tau| < \infty
$. Then, if $a'$ can be chosen so that $-\mathcal{R}e (a_M+1) < a' < -\mathcal{R}e(a_M)$ for some $M$, $f(x)$ has the asymptotic expansion
$$
f(x) \sim  \sum_{m=0}^{M}\sum_{n=0}^{N(m)}A_{mn}x^{a_m}(\log x)^n\; \text{as}\; x \rightarrow 0^{+}.
$$
\end{theorem}
%%%%%%%%%%%%%%%%%%%
\subsection{Tempered and inverse tempered stable subordinators}
The $\alpha$-stable subordinator $D_{\alpha}(t)$ has following Laplace transform (see \cite{Samorodnitsky1994})
\begin{equation}\label{LT_Stable}
\mathbb{E}\left(e^{-s D_{\alpha}(t)}\right) = \mathcal{L}_x\{f_{\alpha}(x,t)\}= e^{- t s^{\alpha}},\;s>0,\;\alpha\in(0,1).
\end{equation}
The Mellin transform of the density of $D_{\alpha}(t)$ (see \cite{Kataria2018})
\begin{align} \label{M_stable}
\mathcal{M}_x\{f_{\alpha}(x,t)\}(s)= F_{\alpha}(s,t) =\frac{ 1}{ \alpha t^{(1-s)/ \alpha}} \frac{\Gamma(s)\Gamma\left((1-s)/\alpha \right)\sin{((1-s)\pi)}}{\pi} .
\end{align}
Using \eqref{LT_Stable}, we have
\begin{align*}
  \mathcal{M}_t\{\mathcal{L}_x\{f_{\alpha}(x,t)\} \}  = \int_{0}^{\infty}e^{- t s^{\alpha}}t^{u-1}du = \frac{\Gamma(u)}{s^{\alpha u}}.
\end{align*}
By inverting the Laplace transform, it leads to
\begin{equation}\label{Mellin_Stable_t}
\mathcal{M}_t\{f_{\alpha}(x,t)\}(u) = \frac{\Gamma(u)}{\Gamma(\alpha u )}x^{\alpha u-1} = \frac{1}{\pi}\Gamma(u)\Gamma(1-\alpha u)\sin(\pi\alpha u)x^{\alpha u-1}.
\end{equation}
The right continuous inverse of $D_\alpha(t)$ defined by $E_{\alpha}(t) = \inf\{w>0: D_{\alpha}(w) > t\}$ is called the inverse stable subordinator. Let $h_{\alpha}(x,t)$ be the density of the inverse stable subordinator. Using similar argument as discussed above, one can write
\begin{equation}\label{Mellin_inverse_Stable_x}
\mathcal{M}_x\{h_{\alpha}(x,t)\}(s) = \frac{1}{t^{(1-s)\alpha}}\frac{\Gamma(s)\Gamma((1-s)\alpha) \sin((1-s)\alpha \pi)}{\pi}.
\end{equation}

The tempered stable subordinator $D_{\alpha, \lambda}(t)$ with tempering parameter $\lambda>0$ and stability index $\alpha \in(0,1)$, has the Laplace transform (see e.g. \cite{Meerschaert2013})
\begin{equation}\label{tem-LT}
\mathbb{E}\left(e^{-s D_{\alpha, \lambda}(t)}\right)=
e^{-t\big((s + \lambda)^{\alpha}-\lambda^{\alpha}\big)}. \end{equation} 
Note that tempered stable distributions are obtained by exponential tempering in the distributions of $\alpha$-stable distributions  (see \cite{Rosinski2007}). The advantage of tempered stable distribution over an $\alpha$-stable distribution is that it has moments of all order and its density is also infinitely divisible.  The probability density function for $D_{\alpha, \lambda}(t)$ is given by 
\begin{equation}\label{ts-density}
 f_{\alpha, \lambda}(x, t)= e^{-\lambda x+\lambda^{\alpha}t} f_{\alpha}(x,t),~~ \lambda>0, \;\alpha\in (0,1), 
\end{equation}
where $f_{\alpha}(x,t)$ is the marginal PDF of an $\alpha$-stable subordinator \cite{Zolotarev1986}. 
The tail probability of positive tempered stable distribution has the following asymptotic behavior
\begin{align}\label{tail-TSS}
\mathbb{P}(D_{\alpha, \lambda}(1)> x) &\sim c_{\beta,\lambda}\frac{e^{-\lambda x}}{x^{\beta}},\;\mbox{as}\;x\rightarrow \infty,
\end{align}
where $c_{\beta,\lambda} = \frac{1}{\beta\pi}\Gamma(1+\beta)\sin(\pi\beta)e^{\lambda^{\beta}}.$
The first two moments of positive tempered stable distribution are given by
\begin{equation}\label{moments-tss}
\mathbb{E}(D_{\alpha, \lambda}(1)) = \beta \lambda^{\beta-1}, \;\; \mathbb{E}(D_{\alpha, \lambda}(1))^2 = \beta(1-\beta) \lambda^{\beta-2} + (\beta \lambda^{\beta-1})^2. 
\end{equation}
The right continuous inverse of $D_{\alpha,\lambda}(t)$ defined by $E_{\alpha,\lambda}(t) = \inf\{w>0: D_{\alpha,\lambda}(w) > t\}$ is called the inverse tempered stable subordinator. Further, we denote here the density of the inverse tempered stable subordinator by $h_{\alpha, \lambda}(x,t)$. Different integral forms of pdf of the inverse tempered stable subordinator are discussed in \cite{Alrawashdeh2017, Kumar2015}. For more properties of inverse subordinators see e.g. \cite{Kumar2019, Leonenko2014}.
 
\section{Densities of stable and inverse stable subordinators}

In this section, we provide the densities of $\alpha$-stable subordinator and inverse stable subordinator in power series form. Note that $\alpha$-stable (one sided positive stable) subordinator density is known in literature in power series as well as in integral form. Here, we recall it using Mellin transform. 
\begin{proposition}\label{Prop_density}
The density of $\alpha$-stable subordinator $D_{\alpha}(t),\; 0 < \alpha < 1$ in series form is
\begin{align}\label{stable_density_1}
     f_{\alpha}(x,t) =  \frac{1}{\pi}\sum_{k=0}^{\infty}\frac{(-1)^{k}}{k!} \Gamma(1+k \alpha)x^{-(1+\alpha k)}t^{k} \sin{(\pi \alpha k)},\;x>0.
\end{align}
\end{proposition}

\begin{proof}
Using \eqref{M_stable}, the inverse Mellin transform is
\begin{align}\label{iterm_cal}
\mathcal{M}^{-1}_{x}\{F_{\alpha}(s,t)\}(x) =f_{\alpha}(x,t)=\frac{1}{2\pi i}\int_{c-i\infty}^{c+i\infty} \frac{ x^{-s}}{ \alpha t^{(1-s)/ \alpha}} \frac{\Gamma(s)\Gamma\left((1-s)/\alpha \right)\sin{((1-s)\pi)}}{\pi}ds.
\end{align}
Denote $g_{\alpha}(s,t)  =  \frac{ x^{-s}}{ \alpha t^{(1-s)/ \alpha}} \frac{\Gamma(s)\Gamma\left((1-s)/\alpha \right)\sin{((1-s)\pi)}}{\pi},\;x>0.
$
The function $g_{\alpha}(s,t)$ is analytic for $\mathcal{R}e(s)<1$ and it has simple poles at $s=1+\alpha k,\; k=0,1,2,\cdots$. Consider a contour which is combination of a straight line parallel to $y$ axis at some point $(c,0),\;c<1$ and a circular arc with radius $R$. As $R\rightarrow\infty$, the integral in \eqref{iterm_cal} equals to the sum of the residues of the integrand with a minus sign since the contour is traversed clockwise. We have 
\begin{align*}
\mathop{res}_{s=1+\alpha k} g_{\alpha}(s,t) &= \lim_{s \to 1+\alpha k } \frac{ x^{-s}}{ \alpha t^{(1-s)/ \alpha}} \frac{\Gamma(s)\Gamma\left((1-s)/\alpha \right)\sin{((1-s)\pi)}}{\pi}\\
&= \frac{1}{\pi}\frac{(-1)^{k}}{k!} \Gamma(1+k \alpha)x^{-(1+\alpha k)}t^{k} \sin{(\pi \alpha k)},
\end{align*}
which completes the proof.
%Thus we have
%$$
%\frac{1}{2\pi i}\int_{c-i\infty}^{c+i\infty} \frac{1 }{\pi \alpha t^{\frac{1}{\alpha}}}h_{\alpha}(s,t) ds = \sum_{k=0}^{\infty}  \mathop{res}_{s=1+\alpha k}h_{\alpha}(s,t).
%$$
%After substituting, we obtain \eqref{stable_density_1}.
\end{proof}
\begin{remark} The density $f_{\alpha}(x,t)$ can also be obtained by using the Mellin transform with respect to the time variable $t$ given in \eqref{Mellin_Stable_t}. Here the function is analytic for $\mathcal{R}e(u)>0$ and there are simple poles at $ u=-k,\;k=0, 1, 2, \cdots.$  
\end{remark}
\begin{proposition}
The density $h_{\alpha}(x,t)$ of the inverse stable subordinator $E_{\alpha}(t)$ is given by the power series
\begin{align}\label{inverse_stable_density_1}
     h_{\alpha}(x,t) = \frac{1}{\pi} \sum_{k=0}^{\infty}\frac{(-x)^{k}}{k!} \Gamma((1+k) \alpha)t^{-(k+1)\alpha} \sin{(\pi \alpha (k+1))},\;x>0.
\end{align}
\end{proposition}
\begin{proof} Using \eqref{Mellin_inverse_Stable_x}, let
\begin{align}
H_{\alpha}(s,t)= \mathcal{M}_x\{h_{\alpha}(x,t)\}(s) = \frac{1}{t^{(1-s)\alpha}}\frac{\Gamma(s)\Gamma((1-s)\alpha) \sin((1-s)\alpha \pi)}{\pi},
\end{align}
which is an analytic function for $\mathcal{R}e(s)>0$ and has simple poles at $s=-k,\;k=0,1,2,\cdots.$
Using the Mellin inversion formula
\begin{align*}
   h_{\alpha}(x,t)= \frac{1}{2 \pi i} \int_{c-i \infty}^{c+i \infty}H_{\alpha}(s,t) x^{-s} ds.
\end{align*}
To evaluate this integral, we consider a contour which is the combination of a straight line parallel to $y$-axis, passing through the point $(c,0),\;c>0$ and an arc of radius $R$ traversing anti-clockwise. The integral along the arc is 0 as $R\rightarrow \infty.$ Thus, the value of the integral is equal to the some of residues at $s=-k,\;k=0,1,2,\cdots$, which is
$$
 \mathop{res}_{s=-k} H_{\alpha}(s,t)x^{-s}= \frac{1}{\pi}\frac{(-x)^{k}}{k!} \Gamma((1+k) \alpha)t^{-(k+1)\alpha} \sin{(\pi \alpha (k+1))}.
$$
\end{proof}
%\begin{remark}
%If we take the singularities at $s=1+\frac{k}{\alpha}$ then the density of inverse stable subordinator,
%$$
%g_{\alpha}^{2}(x,t)= \sum_{k=0}^{\infty}\frac{(-t)^{k}}{k!}\frac{\Gamma(1+k/ \alpha) \sin(k \pi) x^{-(1+k/ \alpha)} }{\pi \alpha}.
%$$
%\end{remark}
\begin{remark}
An alternate proof of above discussed density of ISS can be observed by considering the Inverse Mellin transform with respect to the time variable $t$, where contour is chosen such that it includes poles singularities at $s=(1+\alpha) k,\;k=1,2,\cdots.$
\end{remark}
Next, we provide the pdf of the product of two independent inverse stable subordinators  $E_{\alpha_1}(t)$ and $E_{\alpha_2}(t)$. Let $H_{\alpha_1}(s,t)$ and $H_{\alpha_2}(s,t)$ be the Mellin transfrom of these subordinators. The Mellin transform of the product $E_{\alpha_1,\alpha_2}(t) = E_{\alpha_1}(t)E_{\alpha_2}(t)$ can be written by using the convolution property of Mellin transform and is given by
\begin{align}
H_{\alpha_1,\alpha_2}(s,t)= H_{\alpha_1}(s,t)H_{\alpha_2}(s,t).
\end{align}
 Let $\phi(s)= \frac{\Gamma'(s)}{\Gamma(s)}$ be the digamma function.
\begin{proposition}
The marginal pdf for the product process $E_{\alpha_1,\alpha_2}(t)$ is given by 
\begin{align*}
    h_{\alpha_1,\alpha_2}(x,t)&= \sum_{k=1}^{\infty} \frac{x^{(k-1)}}{\pi^2(k-1)!)^2} \Gamma(k \alpha_1)\Gamma(k \alpha_2)t^{-k(\alpha_1+\alpha_2)} \sin{(\pi \alpha_1 k)}\sin{(\pi \alpha_2 k)}\\
   & \times \left[2\phi(k)-\alpha_1 \phi(k\alpha_1)-\alpha_2 \phi(k\alpha_2)-\log(x)+(\alpha_1+\alpha_2)\log(t)  \right] \\
   & - \frac{1}{2\pi} \sum_{k=1}^{\infty} \frac{x^{(k-1)}}{((k-1)!)^2} \Gamma(k \alpha_1)\Gamma(k \alpha_2)t^{-k(\alpha_1+\alpha_2)}\\
   &\times \left[(\alpha_1+\alpha_2)\sin((\alpha_1+\alpha_2)k\pi)  +(\alpha_1-\alpha_2)\sin((\alpha_1-\alpha_2)k\pi)\right],\;x>0.
\end{align*}
\end{proposition}
\begin{proof}
The density function $h_{\alpha_1,\alpha_2}(x,t)$ of $E_{\alpha_1,\alpha_2}(t)$, using the Mellin inversion formula, is
\begin{align}\label{mellin_inve_pro}
    h_{\alpha_1,\alpha_2}(x,t) = \frac{1}{2 \pi i} \int_{c-i \infty}^{c+i \infty} x^{-s} H_{\alpha_1,\alpha_2}(s,t)  ds
\end{align}
The integrand is analytic for $\mathcal{R}e(s)>0$. Consider a closed contour $\mathcal{C}$: $ABCA$ consisting of the arc $ABC$ which is a circle with radius $R$ having center at $P(0,0)$ and the line segment $CA$ from $c-iR$ to $c+iR$ (see Fig. 1).
\begin{figure}[ht!]
\centering\includegraphics[scale=1.6]{{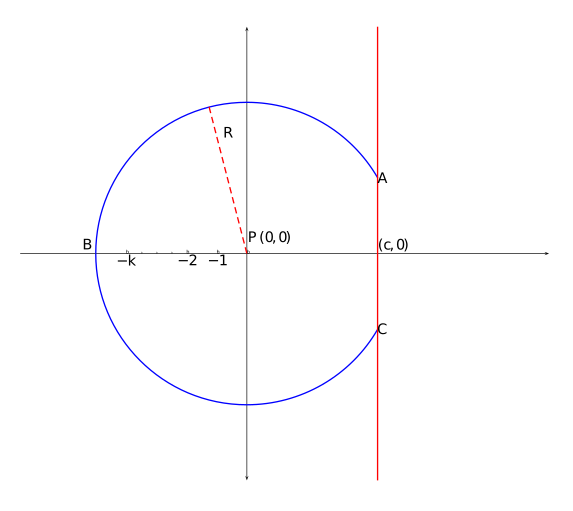}}
\caption{Contour ABC}
\end{figure}
\noindent By Cauchy residue theorem
$$
\int_{c-iR}^{c+iR}x^{-s} H_{\alpha_1,\alpha_2}(s,t)  ds
+ \int_{ABC} x^{-s} H_{\alpha_1,\alpha_2}(s,t) ds= 2\pi i \sum_{k=0}^{n} \mathop{res}_{s=-k} x^{-s}H_{\alpha_1,\alpha_2}(s,t).
$$
The function has $n$ number of poles of second order at $s=-k,\; k=0,1,2,\cdots$, where $n$ denotes greatest integer smaller then $R$. On evaluation, we find that integral along the arc $ABC$ tends to zero  as $R$ goes to $\infty$. So we only calculate the residue at $s=-k$, because the function $(\Gamma(s))^{2}$ appears in the numerator. Along arc $ABC$, we have $s=Re^{i \theta}, \pi/2-\epsilon<\theta<3\pi/2+\epsilon,\;\epsilon>0$, which implies $ds=iRe^{i\theta}$. To estimate the upper bound for $x^{-s} H_{\alpha_1,\alpha_2}(s,t)$, we apply Jordan's Lemma, which yields
\begin{align*}
\left|\int_{ABC} x^{-s} H_{\alpha_1,\alpha_2}(s,t)\right|\leq \pi R \mathop{res}_{s \in ABC}|x^{-s} H_{\alpha_1,\alpha_2}(s,t)|\leq a R^{-R(2-\alpha_1-\alpha_2)} \to 0, \text{as}\;  R \to \infty.
\end{align*}
Where the Stirling formula $|\Gamma(s)| \leq  \sqrt(2\pi)R^{u-1/2}e^{u-vt}$ for $s=u+iv=Re^{i R\theta}$ and real part of $s$ negative,  (see \cite{Karlov2013})
\begin{align*}
|\Gamma(s)|^{2} \leq a_1R^{-2R-1},\; \;  \left| \Gamma((1-u)\alpha_1-v\alpha_1) \right| \leq a_2 R^{R\alpha_1}\\ 
 |\sin{(\pi \alpha_1 (1-s))}| \leq a_3 \cosh(-\alpha_1\pi vt), \; \; |x^{-s}| \leq  a_4 x^R,
\end{align*}
here $a_1, a_2, a_3,a_4$ are real constant. Thus we have
\begin{equation}\label{residue_the}
 \frac{1}{2\pi i }\int_{c-i\infty}^{c+i \infty} x^{-s} H_{\alpha_1,\alpha_2}(s,t) ds= \sum_{k=0}^{\infty} \mathop{res}_{s=-k}x^{-s} H_{\alpha_1,\alpha_2}(s,t).
\end{equation}
 Write $G(s,t)= (s+k)^2x^{-s} H_{\alpha_1,\alpha_2}(s,t)$, which yields
 $$
 G(s,t)= \frac{(\Gamma(s+k+1))^2}{(s+k-1)^2\cdots s^2}\frac{x^{-s}\Gamma\left((1-s)\alpha_1\right)\Gamma\left((1-s)\alpha_2\right)\sin((1-s)\alpha_1\pi)\sin((1-s)\alpha_2\pi)}{\pi^2 t^{(1-s)(\alpha_1+\alpha_2)}}.
 $$
 Using the Binomial theorem (see e.g. 3.6.9 \cite{Abramowitz1964}) and Taylor theorem (see e.g. 3.6.4 \cite{Abramowitz1964}) for expansion of each term in $G(s,t)$ on the neighborhood of point $s=-k$, it follows
 \begin{align*}
    x^{-s}&= x^{k}\left(   1-(s+k)\log x+\cdots \right)\\
    t^{(1-s)(\alpha_1+\alpha_2)} &= t^{(1+k)(\alpha_1+\alpha_2)}\left(  1+(\alpha_1+\alpha_2)(s+k)\log t+\cdots  \right)\\
    (\Gamma(s+k+1))^2&= 1+2\Gamma^{`}(1)(s+k)+\cdots\\
    \Gamma\left((1-s)\alpha_1\right)& = \Gamma\left((1+k)\alpha_1\right)-\alpha_1\Gamma^{`}((1+k)\alpha_1)(s+k)+\cdots\\
    \sin((1-s)\alpha_1\pi)&= \sin((1+k)\alpha_1\pi)-\alpha_1\pi \cos((1+k)\alpha_1\pi)(s+k)+\cdots\\
    (s+m)^{-2} & = \frac{1}{(k-m)^2}+\frac{2}{(k-m)^3}(s+k)+\cdots \; \text{for} \;0\leq m\leq k-1.
 \end{align*}
  Collecting the coefficients associated with the first order term of $(s+k)$ in above series yields
\begin{align*}
  &\frac{x^k}{(k!)^2} \Gamma((k+1) \alpha_1)\Gamma((k+1) \alpha_2)t^{-(k+1)(\alpha_1+\alpha_2)} \sin{(\pi \alpha_1 (k+1))}\sin{(\pi \alpha_2 (k+1))}\\
   & \times \left[2\phi(1)-\alpha_1 \phi((k+1)\alpha_1)-\alpha_2 \phi((k+1)\alpha_2)-\log(x)+(\alpha_1+\alpha_2)\log(t)+\sum_{m=1}^{k}\frac{2}{m} \right] \\
   & - \pi \frac{x^k}{(k)!^2} \Gamma((k+1) \alpha_1)\Gamma((k+1) \alpha_2)t^{-(k+1)(\alpha_1+\alpha_2)}\times\\
   &\left[\alpha_1\cos((1+k)\alpha_1\pi)\sin((1+k)\alpha_2\pi)+\alpha_2\sin((1+k)\alpha_1\pi)\cos((1+k)\alpha_2\pi)\right],
\end{align*}
where $\phi(s)$ is the digamma function. Using the result $\phi(k+1)= \phi(1)+\sum_{m=1}^{k}\frac{1}{m}$ (see  \cite{Abramowitz1964}), the residue of $G(s,t)$ at $s=-k$ is given by
\begin{align*}
    \mathop{res}_{s=-k}G(S,t)&=\frac{x^k}{\pi^2(k!)^2} \Gamma((k+1) \alpha_1)\Gamma((k+1) \alpha_2)t^{-(k+1)(\alpha_1+\alpha_2)} \sin{(\pi \alpha_1 (k+1))}\sin{(\pi \alpha_2 (k+1))}\\
   & \times \left[2\phi(k+1)-\alpha_1 \phi((k+1)\alpha_1)-\alpha_2 \phi((k+1)\alpha_2)-\log(x)+(\alpha_1+\alpha_2)\log(t)\right] \\
    & - \frac{1}{2\pi} \frac{x^k}{(k)!^2} \Gamma((k+1) \alpha_1)\Gamma((k+1) \alpha_2)t^{-(k+1)(\alpha_1+\alpha_2)}\\
   &\times \left[(\alpha_1+\alpha_2)\sin((1+k)(\alpha_1+\alpha_2)\pi)-(\alpha_1-\alpha_2)\sin((1+k)(\alpha_1-\alpha_2)\pi)\right].
\end{align*}
Putting the value of residues in \eqref{residue_the}, proves the desired result.
\end{proof}
%%%%%%%%%%%%%%%%%%%%%%%%%
\noindent Here, we establish the densities of $E_{\alpha}^n(t) = (E_{\alpha}(t))^n$ and $D_{\alpha}^n(t) = (D_{\alpha}(t))^n,\;n\in\mathbb{N}$, which are the $n$-th power of the inverse stable subordinator and $\alpha$-stable subordinator respectively. 
\begin{proposition}
The density function of $E_{\alpha}^n(t)$ is given by
\begin{equation}
h_{\alpha}^n(x,t)  = \sum_{k=1}^{\infty} \frac{(-1)^{k-1}x^{k/n-1}t^{-\alpha k}}{\pi n (k-1)! }\Gamma(\alpha k) \sin{(\alpha k \pi)},\;x>0.
\end{equation}
\end{proposition}
\begin{proof}
The Mellin transform with respect to $x$ is
$$
\mathcal{M}_x\{h_{\alpha}^n(x,t)\} = \frac{\Gamma(n(s-1)+1)\Gamma(\alpha n(1-s))}{\pi}t^{\alpha n (s-1)}\sin{(\alpha n \pi (1-s))},
$$
which is analytic for  $\mathcal{R}e(s)>0$ and has simple poles at $ s = 1-(k+1)/n, \; k= 0,1,2,\ldots$. The density $h_{\alpha}^n(x,t)$ can be obtained by taking the sum of residues of the integrand $x^{-s}\mathcal{M}_x\{h_{\alpha}^n(x,t)\}$ at $s = 1-(k+1)/n, \; k= 0,1,2,\ldots$.
\end{proof}

\begin{proposition}
The pdf $f_{\alpha}^{n}(x,t)$ of $D_{\alpha}^{n}(t)$ is given by
\begin{equation}
     f_{\alpha}^{n}(x,t) =  \frac{1}{n\pi}\sum_{k=0}^{\infty}\frac{(-1)^{k}}{k!} \Gamma(1+k \alpha)x^{-(1+\alpha k/n)}t^{k} \sin{(\pi \alpha kn)},\;x>0.
\end{equation}

\end{proposition}
\begin{proof}
Taking the Mellin transform of $f_{\alpha}^{n}(x,t)$ is
$$
\mathcal{M}_x\{f_{\alpha}^n(x,t)\} = \frac{\Gamma(n(s-1)+1)\Gamma((1-s)n/ \alpha)}{\pi \alpha} t^{n(s-1)/\alpha}\sin{(\pi n (1-s))}.
$$
For evaluating the density, we use the Mellin inversion formula. The function is analytic for $\mathcal{R}e(s)<1$ and has simple poles at $s=1+\alpha k/n, \; k= 0,1,2,\ldots$. Calculating the residues of the integrand at $s=1+\alpha k/n$, give the desired result.
\end{proof}
%%%%%
\noindent Next we discuss the pdf of the product of two independent $\alpha$-stable subordinators $D_{\alpha_1}(t)$ and $D_{\alpha_2}(t)$ with indices $0<\alpha_1, \alpha_2<1$. Let $F_{\alpha_1}(s,t)$ and $F_{\alpha_2}(s,t)$ be the  Mellin transforms of two independent $\alpha$-stable subordinators $D_{\alpha_1}(t)$ and $D_{\alpha_2}(t)$ respectively.

\begin{remark}
The pdf $f_{\alpha_1, \alpha_2}(x,t)$ of the product $D_{\alpha_1,\alpha_2}(t)= D_{\alpha_1}(t)D_{\alpha_1}(t)$ is given by (see \cite{Karlov2013})
\begin{align}
     f_{\alpha_1,\alpha_2}(x,t)&= \sum_{k=1}^{\infty} \frac{x^k}{\pi^2(k!)^2} \Gamma(k/ \alpha_1)\Gamma(k/\alpha_2)t^{-k(1/\alpha_1+1/\alpha_2)} \sin{(\pi \alpha_1 k)}\sin{(\pi \alpha_2 k)}\nonumber\\
   & \times \left[2\phi(k)-\alpha_1 \phi(k/\alpha_1)-\alpha_2 \phi(k/\alpha_2)-\log(x)+(1/\alpha_1+1/\alpha_2)\log(t)  \right] \nonumber\\
   & - \frac{1}{2\pi} \sum_{k=1}^{\infty} \frac{x^(k-1)}{((k)!)^2} \Gamma(k/ \alpha_1)\Gamma(k/ \alpha_2)\nonumber\\
   &\times \left[(\alpha_1+\alpha_2)\sin((\alpha_1+\alpha_2)k\pi)  +(\alpha_1-\alpha_2)\sin((\alpha_1-\alpha_2)k\pi)\right],\;x>0.
\end{align}
\end{remark}

\begin{proposition}
The pdf of the quotient process $E_{\alpha_1}(t)/E_{\alpha_2}(t)$ is 
\begin{align}
   q_{\alpha_1,\alpha_2}(x,t)=  \sum_{k=1}^{\infty} \frac{(-x)^{k-1}}{\Gamma(1-k\alpha_1) \Gamma (1+k\alpha_2)t^{k(\alpha_1-\alpha_2)}},\;x>0.
\end{align}
\end{proposition}

\begin{proof}
Using \eqref{mellin_quiotent}, we have
\begin{align*}
Q_{\alpha_1,\alpha_2}(s,t)&= \mathcal{M}_{x}( q_{\alpha_1,\alpha_2}(x,t))= H_{\alpha_1}(s,t) H_{\alpha_2}(-s+2,t)\\
&= \frac{\Gamma(s)\Gamma(-s+2)\Gamma((1-s)\alpha_1)\Gamma((s-1)\alpha_2)}{t^{(1-s)(\alpha_1-\alpha_2)}}\sin{(1-s)\alpha_1\pi)}\sin{((s-1)\alpha_2\pi)}.
\end{align*}
Using Mellin inversion formula for calculating %$q_{\alpha_1,\alpha_2}(x,t)$,
%$$
% q_{\alpha_1,\alpha_2}(x,t)= \frac{1}{2 \pi i} \int_{c-i \infty}^{c+i %\infty}Q_{\alpha_1,\alpha_2}(s,t) x^{-s} ds.
%$$
Here either we can consider a contour similar to Fig. 1, enclosing the simple poles at $s=-k$ or take a contour enclosing the simple poles at $s=k+2$. It follows
\begin{align*}
&\mathop{res}_{s=-k} Q_{\alpha_1,\alpha_2}(s,t)x^{-s} \\
&=  \frac{(-1)^{k+1}x^{k}\Gamma(k+2) \Gamma ((1+k)\alpha_1)\Gamma(-(1+k)\alpha_2) \sin{((1+k)\alpha_1)}\sin{((1+k)\alpha_2)}}{k!t^{(1+k)(\alpha_1-\alpha_2)}}.
\end{align*}
The result follows after using the Euler's reflection formula and taking the sum of residues.
\end{proof}
\section{Tempered stable and inverse tempered stable subordinators}
In this section, we obtain the pdf and its asymptotic behaviour of the inverse stable subordinator using Mellin transform.
\begin{proposition}
The density $f_{\alpha, \lambda}(x, t)$ of tempered stable subordinator $D_{\alpha, \lambda}(t)$ is given by
\begin{align}
    f_{\alpha, \lambda}(x, t) &= e^{-\lambda x+t \lambda^{\alpha}}\sum_{k=0}^{\infty}\frac{(-1)^{k}}{\pi k!} \Gamma(1+k \alpha)x^{-(1+\alpha k)}t^{k} \sin{(\pi \alpha k)},\;x>0.
\end{align}
\end{proposition}
\begin{proof}
Let $\mathcal{L}_{x}(f_{\alpha, \lambda}(x, t))$ be the Laplace transform of the pdf $f_{\alpha, \lambda}(x, t)$ with respect to the $x$ variable. We have 
\begin{align}
    \mathcal{L}_{x}(f_{\alpha, \lambda}(x, t)) = e^{-t((s+\lambda)^{\alpha}-\lambda^{\alpha})}.
\end{align}
 Using  Mellin transform with respect to time variable $t$, yields
\begin{align}\label{laplace_density}
 \mathcal{M}_{t}[\mathcal{L}_{x}(f_{\alpha, \lambda}(x, t))] = \int_{0}^{\infty} t^{u-1}e^{-t((s+\lambda)^{\alpha}-\lambda^{\alpha})} dt = \frac{\Gamma(u)}{((s+\lambda)^{\alpha}-\lambda^{\alpha})^{u}}.
\end{align}
For inverting the Laplace transform in \eqref{laplace_density}, the generalized Mittag-Leffler function is required which is introduced here. The generalized Mittag-Leffler function, introduced by \cite{Prabhakar1971}, is defined by
\begin{equation}
M_{a,b}^c(z) = \sum_{n=0}^{\infty}\frac{\Gamma(c+n)}{\Gamma(c)\Gamma(a n + b)}\frac{z^n}{n!},
\end{equation}
where $a, b, c \in \mathbb{C}$ with $\mathcal{R}(b) >0$. When $c = 1$, it reduces to Mittag-Leffler function. Further,
\begin{equation}\label{gml-0}
M_{a,b}^c(0) = \frac{(c)_0}{\Gamma{(b)}} = \frac{1}{\Gamma{(b)}}.
\end{equation}
The function $F(s) = \frac{s^{ac - b}}{(s^a+ \eta)^c}$ has the inverse LT \cite{Kumar2018}
\begin{equation}\label{lt-for-special-func}
\mathcal{L}^{-1}[F(s)] = x^{b-1}M_{a,b}^{c}(-\eta x^a).
\end{equation}
Moreover, using the shifting property of Laplace transform and equation \eqref{lt-for-special-func}, we have
\begin{equation}\label{Ilt-for-special-func}
\mathcal{L}^{-1}\left[\frac{1}{((s+\lambda)^{\alpha}-\lambda^{\alpha})^{u}}\right] =e^{-\lambda x}x^{\alpha u-1}M_{\alpha, \alpha u}^{u}(\lambda^{\alpha}x^{\alpha}),
\end{equation}
which follows by taking $a=\alpha, b= \alpha u, c=u $ and $\eta =-\lambda^{\alpha}$. Now by inverting the LT in \eqref{laplace_density} with the help of \eqref{Ilt-for-special-func}, it follows
$$
\mathcal{M}_{t}(f_{\alpha, \lambda}(x, t)) = e^{-\lambda x} x^{\alpha u-1} \sum_{n=0}^{\infty}\frac{\Gamma(u+n)}{\Gamma(\alpha n+\alpha u)} \frac{(\lambda x)^{\alpha n}}{n!}.
$$
By inverting the Mellin transform with poles at $s=-(n+k)$ leads to
\begin{align*}
 f_{\alpha, \lambda}(x, t)& = \frac{1}{2\pi i}\int_{c-i\infty}^{c+i\infty} t^{-u}e^{-\lambda x} x^{\alpha u-1} \sum_{n=0}^{\infty}\frac{\Gamma(u+n)}{\Gamma(\alpha n+\alpha u)} \frac{(\lambda x)^{\alpha n}}{n!}du\\
&= e^{-\lambda x+\lambda^{\alpha}t}\sum_{k=0}^{\infty} \frac{(-1)^{k+1}t^{k}x^{-(\alpha k+1)}}{k!\Gamma(-\alpha k)},\;x>0.
 \end{align*}
The result follows by using the relationship $\Gamma(z)\Gamma(1-z) = \frac{\pi}{\sin(\pi z)}.$
\end{proof}
\begin{proposition} The pdf of the inverse tempered stable subordinator $E_{\alpha, \lambda}(t)$ is given by the following power series representation,
\begin{align*}
    h_{\alpha, \lambda}(x,t) = \sum_{k=0}^{\infty} e^{-\lambda t+\lambda^{\alpha}x}\frac{(-x)^{k} t^{-\alpha(k+1)}}{k!}\sum_{m=0}^{\infty}\left[  \frac{(t\lambda)^{m}}{\Gamma(-\alpha(k+1)+m+1))}- \frac{t^{m}\lambda^{\alpha+m}}{\Gamma(-\alpha k+1+m)}\right].
\end{align*}
\end{proposition}
\begin{proof}
The Laplace transform of the pdf $h_{\alpha, \lambda}(x,t)$ with respect to the time variable $t$ is (see \cite{Meerschaert2008})
$$
\mathcal{L}_{t}(f_{\alpha, \lambda}(x, t)) = \frac{1}{u}((u+\lambda)^{\alpha}-\lambda^{\alpha})e^{-x((u+\lambda)^{\alpha}-\lambda^{\alpha})}.
$$
By taking Mellin transform with respect to $x$, we have 
\begin{align}\label{Laplace_ITSS}
    \mathcal{M}_{x}[\mathcal{L}_{t}(f_{\alpha, \lambda}(x, t))] &= \int_{0}^{\infty} x^{s-1}
    \frac{1}{u}((u+\lambda)^{\alpha}-\lambda^{\alpha})e^{-x((u+\lambda)^{\alpha}-\lambda^{\alpha})} dx \nonumber\\
    & = \Gamma(s)\left[ \frac{(u+\lambda)^{\alpha}}{u((u+\lambda)^{\alpha}-\lambda^{\alpha})^{s}}-\frac{\lambda^{\alpha}}{u((u+\lambda)^{\alpha}-\lambda^{\alpha})^{s}}\right].
\end{align}
For evaluating the inverse Laplace transform in \eqref{Laplace_ITSS},
suppose $F_1(u)=\frac{u^{\alpha}}{(u^{\alpha}-\lambda^{\alpha})^{s}}$ and $F_2(u)=\frac{\lambda^{\alpha}}{(u^{\alpha}-\lambda^{\alpha})^{s}}$. Using \eqref{lt-for-special-func} and applying the shifting property of Laplace transform with $G(u)=F_1(u+\lambda) - F_2(u+\lambda)$ leads to the inverse Laplace transform 

$$\mathcal{L}^{-1}[G(u)] =e^{-\lambda t}\left[t^{\alpha(s-1)-1} M^{s}_{\alpha,\alpha (s-1)}(\lambda^{\alpha} t^{\alpha})- t^{\alpha s-1} M^{s}_{\alpha,\alpha s}(\lambda^{\alpha} t^{\alpha}) \right].$$ 
Further,
\begin{align*}
\mathcal{L}^{-1}\left[\frac{G(u)}{u}\right]= \int_{0}^{t}e^{-\lambda y}\left(y^{\alpha(s-1)-1} M^{s}_{\alpha,\alpha (s-1)}(\lambda^{\alpha} y^{\alpha})- y^{\alpha s-1} M^{s}_{\alpha,\alpha s}(\lambda^{\alpha} y^{\alpha}) \right)dy. 
\end{align*}
Next, using the property from  \cite{Kilbas2004}, 
\begin{align}
\int_{0}^{t}y^{\mu-1} M_{\rho,\mu}^{\nu}(w y^{\rho})(t-y)^{\nu-1}dy=\Gamma({\nu}) t^{\nu+\mu-1} M_{\rho,\mu+\nu}^{\nu}(w t^{\rho}),
\end{align}
it follows
$$\mathcal{L}^{-1}{\left[\frac{G(u)}{u}\right]}=e^{-t \lambda}\sum_{m=0}^{\infty}\lambda^{m} t^{\alpha s+m}\left(t^{-\alpha} M^{s}_{\alpha,\alpha (s-1)+m+1}(\lambda^{\alpha} t^{\alpha})- \lambda^{\alpha} M^{s}_{\alpha,\alpha s+m+1}(\lambda^{\alpha} t^{\alpha})\right).$$
Putting in \eqref{Laplace_ITSS}, we have
\begin{align}\label{Mellin_ITSS}
 &\mathcal{M}_{x}((f_{\alpha, \lambda}(x, t))) \nonumber\\
 &= e^{-t \lambda}\sum_{m=0}^{\infty}\lambda^{m} t^{\alpha s+m} \sum_{n=0}^{\infty}\left( t^{-\alpha}\frac{\Gamma(s+n)}{\Gamma(\alpha n+\alpha (s-1)+m+1)}\frac{(\lambda t)^{n}}{n!}-  \frac{\lambda^{\alpha}\Gamma(s+n)}{\Gamma(\alpha n+\alpha s+m+1)}\frac{(\lambda t)^{n}}{n!}\right).
\end{align}
 Note that the function in \eqref{Mellin_ITSS} is analytic for $\mathcal{R}e(s)>-n$ and has simple poles at $ s = -n-k, \; k =0,1,2,\ldots$. Considering a similar contour as given in Fig. 1, the result follows by taking the some of residues of the integrand at $ s = -n-k, \; k=0,1,2,\ldots$.
 %\begin{align*}
  %   f_{\alpha, \lambda}(x, t)&=\frac{1}{2\pi i}\int_{c-i\infty}^{c+\infty}x^{-s}e^{-t \lambda}\sum_{m=0}^{\infty}\lambda^{m} t^{\alpha u+m}\left(\sum_{n=0}^{\infty} t^{-\alpha}\frac{\Gamma(u+n)}{\Gamma(\alpha n+\alpha (u-1)+m+1}\frac{(\lambda t)^{n}}{n!}-  \frac{\lambda^{\alpha}\Gamma(u+n)}{\Gamma(\alpha n+\alpha u+m+1}\frac{(\lambda t)^{n}}{n!}\right)\\
   %  &= \sum_{k=0}^{\infty}  \mathop{res}_{s=-n-k}
 %\end{align*}
\end{proof}
Next, we discuss about the the Mellin transform of the inverse Gaussian process and its first exit times. Note that inverse Gaussian process is the hitting time process of the Brownian motion with drift and is defined by (see \cite{Applebaum2009}, p.54)
$$
G(t) = \inf\{w>0: B(w) + \gamma w> \delta t\},
$$
where $B(t)$ is the standard Brownian motion. Let $g_{\delta, \gamma}(x,t)$ be the pdf of $G(t)$. Then the Laplace transform of $g_{\delta, \gamma}(x,t)$ with respect to $x$ is
\begin{equation}\label{LT_IG}
\mathcal{L}_{x}(g_{\delta, \gamma}(x,t))= e^{-\delta t(\sqrt{\gamma^2+2s}-\gamma)}.
\end{equation}
\begin{proposition}
The Mellin transform of the pdf $g_{\delta, \gamma}(x,t)$ with respect to the time variable $t$ is given by
\begin{align}\label{IG_Mellin}
\mathcal{M}_{t}(g_{\delta, \gamma}(x,t)) = \frac{1}{\delta^u 4^{u}}e^{-\gamma^2/2 x}x^{u/2-1} \sum_{n=0}^{\infty}\frac{\Gamma(u+n)}{\Gamma(n/2+u/2)}\frac{(\gamma x^{1/2})^{n}}{4^n n!}.
\end{align}
\end{proposition}
\begin{proof}
We use similar step as for the density of TSS. First we take Laplace transform with respect to $x$ and then the Mellin transform with respect  to  $t$. It follows
$$
\mathcal{M}_{t}(\mathcal{L}_{x}(g_{\delta, \gamma}(x,t))) = \frac{\Gamma(u)}{\delta^u 4^{u}(\sqrt{\frac{\gamma^2}{2}+s}-\frac{\gamma}{2})^{u}}.
$$
For calculating the Laplace  inversion we  use  shifting  property and equation \eqref{Ilt-for-special-func}, which yield
\begin{align*}
\mathcal{M}_{t}(g_{\delta, \gamma}(x,t))&=\frac{\Gamma(u)}{\delta^u 4^{u}}e^{-\gamma^2/2 x} x^{u/2-1} M_{1/2,u/2}^{u}(\frac{\gamma}{4} x^{1/2})\\
&= \frac{1}{\delta^u 4^{u}}e^{-\gamma^2/2 x}x^{u/2-1} \sum_{n=0}^{\infty}\frac{\Gamma(u+n)}{\Gamma(n/2+u/2)}\frac{(\gamma x^{1/2})^{n}}{4^n n!}.
\end{align*}
\end{proof}
\begin{remark}
In \eqref{IG_Mellin}, taking Mellin inversion formula by consider a contour similar to Fig. 1, where the integrand has poles at $s = -n-k,\; k = 0,1,2,\ldots$, we have
\begin{align}\label{pdf_ing}
g_{\delta, \gamma}(x,t)&= e^{-\gamma^2/2 x+\gamma\delta t}\sum_{k=0}^{\infty}\frac{(-1)^k (t\delta x^{-1/2})^{k}}{k!\Gamma(-k/2) x}, \; \text{where k is odd}.\nonumber\\
    & = e^{-\gamma^2/2 x+\gamma\delta t}x^{-3/2} t\delta\sum_{k=0}^{\infty}\frac{(-1)^{2k+1}}{(2k+1)!\Gamma(-k-1/2)}\left(\frac{t^2\delta^2}{x}\right)^{k},\;  k \in \mathbb{N}_{0}.
\end{align}
Using the property $\Gamma(2z)= \pi^{-1/2}2^{2z-1}\Gamma(z)\Gamma(z+1/2)$ and $\Gamma(z)\Gamma(-z)= \frac{\pi}{\sin{\pi z}}$. Then we put $\Gamma(-k-1/2)\Gamma(2(k+1))= (-1)^{k}2^{2k+1}\sqrt{\pi}k!$ in \eqref{pdf_ing}, which leads to
\begin{align}
    g_{\delta, \gamma}(x,t)= (2\pi)^{-1/2}\delta tx^{-3/2}e^{\delta\gamma t}e^{-\frac{1}{2}\left(\frac{\delta^2t^2}{x}+\gamma^2x\right)},\;x>0.
\end{align}
\end{remark}

Let $Q(t) = \inf\{w>0: G(w) > t\}$ be the first-exit time of inverse Gaussian process and $h_{\delta, \gamma}(x,t)$ be pdf of process $Q(t)$ (see \cite{Vellaisamy2018}). We have
\begin{align*}
\mathcal{M}_{x}(h_{\delta, \gamma}(x,t)) &= \frac{\Gamma(s)}{(4\delta)^{s-1}}e^{-\gamma^2/2 t} \sum_{m=0}^{\infty}(\gamma^2/2)^m t^{(s/2+m)}\nonumber\\
&\times\left[t^{-1/2}M_{1/2, (s+1)/2+m}^{s}( t^{1/2}\gamma/4) - \frac{\gamma}{4} M_{1/2, s/2+1+m}^{s}(t^{1/2}\gamma/4 )\right].
\end{align*}
\begin{proposition}
The density $h_{\delta, \gamma}(x,t)$ of the first-exit time of inverse  Gaussian process $Q(t)$ is 
\begin{align}
    h_{\delta, \gamma}(x,t)&= \sum_{k=0}^{\infty} e^{-\gamma^2/2 t+\delta \gamma x}\frac{(-x)^{k} t^{-k/2}}{k!}(4\delta)^{k-1}\nonumber\\
    &\times\sum_{m=0}^{\infty}\left[  \frac{(t\gamma^2/2)^{m}t^{-1/2}}{\Gamma(-k/2+m+1/2))}- \frac{t^{m}\gamma^{2m+1}}{2^{m+2}\Gamma(-k/2+1+m)}\right].
\end{align}
\end{proposition}
\begin{proof}
It follows using the same approach as for the density of inverse tempered stable subordinator.
\end{proof}
%\subsection{Asymptotic Behaviour of Densities}
\begin{remark}
Note that for $x >0$,
$$
\{Q(t) \leq x\} = \{G(x)\geq t\}=\{\mathop{\sup}_{s\leq t} (B(s) +\gamma s)\leq \delta x \}= \{\delta^{-1 } \mathop{\sup}_{s\leq t} (B(s) +\gamma s) \leq x \}
$$
and hence the density function of $Q(t)$ can also be viewed as density of supremum of Brownian motion with drift.
\end{remark}

\begin{proposition}
The asymptotic behaviour of $h_{\alpha, \lambda}(x,t)$ as $ x \to 0^+$ is
$$
\lim_{x \to 0^{+}} h_{\alpha, \lambda}(x,t) = e^{-t\lambda}t^{-\alpha}\sum_{m=0}^{\infty}\frac{(t\lambda)^{m}}{\Gamma(-\alpha+m+1)}-\lambda^{\alpha}t^{-\alpha}
$$
\end{proposition}
\begin{proof}
Using the Mellin transform given in \eqref{Mellin_ITSS}, which is analytic for $\mathcal{R}(s)\leq -n$, and has simple poles at $s=-(n+k), \; k=0,1,2,\ldots$.
The result follows using Theorem \eqref{Asym_mellin}.
\end{proof}
\begin{remark}
For $\lambda=0$, it follows that $\lim_{x \to 0^{+}} h_{\alpha, 0}(x,t) = \frac{t^{-\alpha}}{\Gamma(1-\alpha)}$ see e.g. \cite{Hahn2011}.
\end{remark}

\noindent Again using the application of Theorem \eqref{Asym_mellin}, the following result follows.

\begin{proposition} We have
\begin{enumerate}[A.]
\item 
\begin{align*}
\lim_{x \to 0^{+}} h_{\alpha_1,\alpha_2}(x,t)&=  \frac{t^{-(\alpha_1+\alpha_2)}\log(x)}{ \Gamma( 1-\alpha_1)\Gamma(1- \alpha_2)} \left[2-\alpha_1 \phi(\alpha_1)-\alpha_2 \phi(\alpha_2)-\log(x)+(\alpha_1+\alpha_2)\log(t)  \right] \\
   & - \frac{\log(x)}{2\pi}  \Gamma( \alpha_1)\Gamma( \alpha_2)t^{-(\alpha_1+\alpha_2)} \left[(\alpha_1+\alpha_2)\sin((\alpha_1+\alpha_2)\pi)  +(\alpha_1-\alpha_2)\sin((\alpha_1-\alpha_2)\pi)\right].
\end{align*}
\item 
\begin{align}
 \lim_{x \to 0^{+}} f_{\alpha_1,\alpha_2}(x,t)&=  \frac{\log(x)}{\pi^2} \Gamma(1/ \alpha_1)\Gamma(1/\alpha_2)t^{-(1/\alpha_1+1/\alpha_2)} \sin{(\pi \alpha_1 )}\sin{(\pi \alpha_2 )}\nonumber\\
   & \times \left[2-\alpha_1 \phi(1/\alpha_1)-\alpha_2 \phi(1/\alpha_2)-\log(x)+(1/\alpha_1+1/\alpha_2)\log(t)  \right] \nonumber\\
   & - \frac{\log(x)}{2\pi}  \Gamma(1/ \alpha_1)\Gamma(1/ \alpha_2) \left[(\alpha_1+\alpha_2)\sin((\alpha_1+\alpha_2)\pi)  +(\alpha_1-\alpha_2)\sin((\alpha_1-\alpha_2)\pi)\right].
\end{align}
\item 
\begin{align}
    \lim_{x \to 0^{+}} h_{\alpha}^n(x,t)  =  \frac{(-1)^{n-1}t^{-\alpha n}}{\pi n! }\Gamma(\alpha n) \sin{(\alpha n\pi)}.
\end{align}
\item 
\begin{align}
     \lim_{x \to 0^{+}} q_{\alpha_1,\alpha_2}(x,t)=   \frac{1 }{\Gamma(1-\alpha_1) \Gamma (1+\alpha_2)t^{(\alpha_1-\alpha_2)}}.
\end{align}
\end{enumerate}
\end{proposition}
\noindent {\bf Acknowledgments:} 
NG would like to thank Council of Scientific and Industrial Research(CSIR), India, for the award of a research fellowship. Further, AK would like to express his gratitude to the Science and Engineering Research Board (SERB), India, for the financial support under the MATRICS research grant MTR/2019/000286.
%N. Leonenko was supported in particular by Australian Research Council's Discovery Projects funding scheme (project DP160101366)and  by project MTM2015-71839-P of MINECO, Spain (co-funded with FEDER funds).

\end{document}